\documentclass[]{elsarticle}
\usepackage{amssymb,amsmath,amsthm,soul,color}
\usepackage{t1enc}
\usepackage[utf8]{inputenc}
\usepackage{a4,indentfirst,latexsym}
\usepackage{graphics}
\usepackage{mathrsfs}

\usepackage{hyperref}

\usepackage{todonotes}
\biboptions{numbers,sort&compress}

\input xy
\xyoption{all}

\newtheorem{Th}{Theorem}[section]

\newtheorem{Lem}[Th]{Lemma}

\newtheorem{Rem}[Th]{Remark}

\makeatletter
    
    \newcommand{\Rmnum}[1]{\expandafter\@slowromancap\romannumeral #1@}

\newcommand{\cC}{{\mathcal C}}

\newcommand{\cF}{{\mathcal F}}

\newcommand{\cI}{{\mathcal I}}
\newcommand{\cJ}{{\mathcal J}}
\newcommand{\J}{{\mathcal J}}

\newcommand{\cN}{{\mathcal N}}

\newcommand{\Ga}{\Gamma}

\newcommand{\weakto}{\rightharpoonup}

\newcommand{\R}{\mathbb{R}}
\newcommand{\Z}{\mathbb{Z}}

\newcommand{\supp}{\mathrm{supp}\,}

\numberwithin{equation}{section}

\DeclareMathOperator*{\essinf}{ess\,inf}

\begin{document}

\title{The fractional Schr\"odinger equation with Hardy-type potentials and sign-changing nonlinearities}

\author{Bartosz Bieganowski}
\ead{bartoszb@mat.umk.pl}
\address{Nicolaus Copernicus University, Faculty of Mathematics and Computer Science, ul. Chopina 12/18, 87-100 Toru\'n, Poland}

\begin{abstract} 
We look for solutions to a fractional Schr\"odinger equation of the following form
$$
(-\Delta)^{\alpha / 2} u + \left( V(x) - \frac{\mu}{|x|^{\alpha}} \right) u = f(x,u)-K(x)|u|^{q-2}u\hbox{ on }\R^N \setminus \{0\},
$$
where $V$ is bounded and close-to-periodic potential and $- \frac{\mu}{|x|^{\alpha}}$ is a Hardy-type potential. We assume that $V$ is positive and $f$ has the subcritical growth but not higher than $|u|^{q-2}u$. If $\mu$ is positive and small enough we find a ground state solution, i.e. a critical point of the energy being minimizer on the Nehari manifold. If $\mu$ is negative we show that there is no ground state solutions. We are also interested in an asymptotic behaviour of solutions as $\mu \to 0^+$ and $K \to 0$.
\end{abstract}


\maketitle

\section{Introduction}
\setcounter{section}{1}

We consider the following nonlinear, fractional Schr\"odinger equation with external, Hardy-type potential
\begin{eqnarray}
\label{eq}
(-\Delta)^{\alpha / 2} u + \left( V(x) - \frac{\mu}{|x|^{\alpha}} \right) u = f(x,u)-K(x)|u|^{q-2}u\hbox{ on }\R^N \setminus \{0\}
\end{eqnarray}
where $\alpha \in (0,2)$, $\mu \in \R $ and $N > \alpha$, with $ u \in H^{\alpha / 2}(\R^N)$. The fractional Schr\"odinger equation arises in many models from mathematical physics, e.g. nonlinear optics, quantum mechanics, nuclear physics (see e.g. \cite{deOliveira, GuoXu, Kowalski, Longhi, Luchko, Robinett, Stickler, Tare, ZhangLiu, ZhangZhong} and references therein). We focus on the external potential of the form $V(x) - \frac{\mu}{|x|^\alpha}$, where $V \in L^\infty (\R^N)$ is close-to-periodic potential and $- \frac{\mu}{|x|^\alpha}$ is Hardy-type potential. Note that the Hardy-type potential does not belong to the Kato's class, hence it is not a lower order perturbation of the operator $-\Delta+V(x)$ (see \cite{ReedSimon}).

The fractional Laplacian can be defined via Fourier multiplier $|\xi|^\alpha$, i.e. the operator $(-\Delta)^{\alpha / 2}$, for a function $\psi : \R^N \rightarrow \R$, is given by the Fourier transform by the formula
$$
\cF \left( (-\Delta)^{\alpha / 2} \psi \right) (\xi) := | \xi |^\alpha \hat{\psi} (\xi),
$$
where
$$
\cF \psi (\xi) := \hat{\psi} (\xi) := \int_{\R^N} e^{- i \xi \cdot x} \psi(x) \, dx
$$
denotes the usual Fourier transform. When $\psi : \R^N \rightarrow \R$ is rapidly decaying smooth function, it can be defined by the principal value of the singular integral
\begin{equation}
(-\Delta)^{\alpha / 2} \psi (x) = c_{N, \alpha} P.V. \int_{\R^N} \frac{\psi(x) - \psi(y)}{|x-y|^{N+\alpha}} \, dy,
\end{equation}
where 
$$
c_{N,\alpha} := \frac{2^\alpha \Ga \left( \frac{N+\alpha}{2} \right)}{2 \pi^{N/2} | \Ga (-\alpha / 2) |} > 0.
$$ 
Here, $\Gamma$ denotes the Gamma function, i.e. a function defined for complex numbers $z$ with $\mathrm{Re}(z) > 0$ by the formula
$$
\Gamma(z) := \int_0^\infty x^{z-1} e^{-x} \, dx
$$
and extended to a meromorphic function on the set $\mathbb{C} \setminus \{0, -1, -2, \ldots \}$. Both definitions of the fractional Laplacian are equivalent, i.e. on $L^2(\R^N)$ they give operators with common domain and they coincide on this domain (see \cite{Kwasnicki}). It is known that $(-\Delta)^{\alpha / 2}$ reduces to $-\Delta$ as $\alpha \to 2^-$ (see \cite{DiNezza}). In this paper we identify $(-\Delta)^{\alpha / 2}$ with the classical Laplace operator $-\Delta$ for $\alpha = 2$. In what follows we will use the following characterization of the fractional Sobolev space, for $0 < \alpha < 2$:
$$
H^{\alpha / 2} ( \R^N) := \left\{ u \in L^2 (\R^N) \ : \ \iint_{\R^N \times \R^N} \frac{|u(x)-u(y)|^2}{|x-y|^{N+\alpha}} \, dx \, dy + \int_{\R^N} |u(x)|^2 \,dx < \infty \right\}
$$
with the associated scalar product:
$$
H^{\alpha / 2} ( \R^N) \times H^{\alpha / 2} ( \R^N) \ni (u,v) \mapsto \iint_{\R^N \times \R^N} \frac{(u(x)-u(y))(v(x)-v(y))}{|x-y|^{N+\alpha}} \, dx \, dy + \int_{\R^N} u(x)v(x) \,dx \in \R.
$$
See e.g. \cite{Cabre, DiNezza} for more background about the fractional Laplace operator and fractional Sobolev spaces.

Equation (\ref{eq}) describes the behaviour of the so-called standing wave solutions $\Phi(x,t)=u(x)e^{-i\omega t}$ of the following time-dependent fractional Schr\"odinger equation
$$
i \frac{\partial \Phi}{\partial t} = (-\Delta)^{\alpha / 2} \Phi + \left(V(x) - \frac{\mu}{|x|^\alpha} +\omega\right) \Phi - g(x, |\Phi|).
$$
Such an equation was introduced by Laskin by expanding the Feynman path integral from the Brownian-like to the L\'evy-like quantum mechanical paths (see \cite{Laskin2000, Laskin2002}). The time-dependent equation is also intensively studied (see e.g. \cite{Guedes, LiemertKienle}).

The nonlinearity $f$ satisfies the following conditions:
\begin{itemize}
\item[(F1)] $f :\R^N\times\R \to \R$ is measurable, $\Z^N$-periodic in $x\in\R^N$ and continuous in $u \in\R$ for a.e. $x\in\R^N$, moreover there are $c >0$ and $2< q < p < 2^*_{\alpha} := \frac{2N}{N-\alpha}$ such that
$$|f(x,u)| \leq c (1+|u|^{p-1})\hbox{ for all }u \in\R,\; x\in\R^N,$$
\item[(F2)] $f(x,u)=o(|u|)$ uniformly in $x$ as $|u|\to 0^+$,
\item[(F3)] $F(x,u)/|u|^{q}\to\infty$ uniformly in $x$ as $|u|\to\infty$,
\item[(F4)] $u \mapsto f(x,u)/|u|^{q-1}$ is increasing on $(-\infty,0)$ and on $(0, \infty)$.
\end{itemize} 
Note that our conditions imply that for every $\varepsilon > 0$ there is $C_\varepsilon > 0$ such that
\begin{equation}\label{eps}
|f(x,u)u| \leq \varepsilon |u|^2 + C_\varepsilon |u|^p.
\end{equation}

To study the asymptotic behaviour of solutions we need additionaly the following assumption:
\begin{itemize}
\item[(F5)] for a.e. $x \in \R^N$ the function $f(x, \cdot)$ is of $\cC^1$-class and there are $2 < r \leq p$ and $c > 0$ such that
$$
f(x,u)u - 2 F(x,u) \geq b |u|^r \quad \mbox{and} \quad |f_u ' (x,u)| \leq c (1 + |u|^{p-2})
$$
for all $u \in \R$ and a.e. $x \in \R^N$.
\end{itemize}

We impose on $K$ the following condition: 
\begin{itemize}
\item[($K$)] $K \in L^{\infty}(\R^N)$ is $\Z^N$-periodic in $x\in\R^N$, $K(x)\geq 0$ for a.e. $x\in\R^N$.
\end{itemize} 

In what follows the external potential $V$ satisfies:
\begin{itemize}
\item[($V$)] $V \in L^{\infty}(\R^N)$ is the sum $V = V_{loc} + V_{per}$, where $V_{per} \in L^\infty(\R^N)$ is $\Z^N$-periodic in $x\in\R^N$ and $V_{loc}(x) \to 0$ as $|x|\to\infty$; moreover 
$$
V_0 := \essinf_{x\in\mathbb{R}^N} V(x) > 0.
$$
Moreover $V_{loc} \in L^s (\R^N)$, where $s \geq \frac{N}{\alpha}$.
\end{itemize} 
In view of ($V$) the following norm
\begin{equation}\label{Eq:Norm}
\| u\|^2 := c_{N,\alpha} \iint_{\R^N \times \R^N} \frac{|u(x)-u(y)|^2}{|x-y|^{N+\alpha}} \, dx \, dy + \int_{\R^N} V(x) |u(x)|^2 \,dx
\end{equation}
is equivalent to the classic one in $H^{\alpha / 2}(\R^N)$. Moreover, if $V_{loc} = 0$, $\| \cdot \|$ is $\mathbb{Z}^N$-invariant, i.e. 
$$
\| u \| = \| u (\cdot - z ) \|
$$
for any $z \in \mathbb{Z}^N$.

Recall that $u \in H^{\alpha / 2}(\R^N)$ is a {\it weak solution} to \eqref{eq} if for every function $\varphi \in \cC_0^\infty (\R^N)$ there holds
\begin{align*}
c_{N,\alpha} \iint_{\R^N\times \R^N} \frac{(u(x)-u(y))(\varphi(x)-\varphi(y))}{|x-y|^{N+\alpha}} \, dx \, dy &+ \int_{\R^N} V(x) u(x) \varphi(x) \, dx - \mu \int_{\R^N} \frac{u(x)\varphi(x)}{|x|^\alpha} \, dx \\ &- \int_{\R^N} \left( f(x,u)\varphi - K (x) |u|^{q-2}u\varphi \right) \, dx = 0.
\end{align*}
It is classical to check that the functional $\J:H^{\alpha / 2} (\R^N)\to\R$ given by
$$
\J(u):= \frac{c_{N,\alpha}}{2} \iint_{\R^N \times \R^N} \frac{|u(x)-u(y)|^2}{|x-y|^{N+\alpha}} \, dx \, dy + \frac12 \int_{\R^N} V(x)|u(x)|^2 \,dx - \frac12 \int_{\R^N} \frac{\mu}{|x|^{\alpha}} |u(x)|^2 \, dx - \cI(u),
$$
where
$$
\cI(u) := \int_{\R^N} \Big(F(x,u(x))- \frac{1}{q}K(x)|u(x)|^{q}\Big)\, dx,
$$
is of $\cC^1$-class and critical points of $\cJ$ are solutions of \eqref{eq}. We are looking for a {\em groud state} solution, i.e. a critical point being a minimizer of $\J$ on the Nehari manifold 
\begin{equation}\label{Eq:NehariManifold}
\cN:=\{u\in H^{\alpha / 2} (\R^N) \setminus\{0\}:\; \J'(u)(u)=0\}.
\end{equation}
Obviously $\cN$ contains all nontrivial critical points, hence a ground state is the least energy solution.

The classical Schr\"odinger equation (the case $\alpha = 2$) has been studied by many authors; see for instance \cite{AlamaLi, BieganowskiMederski, BuffoniJeanStuart, CotiZelati, dAveniaMederski, KryszSzulkin, Rabinowitz:1992, Mederski, MederskiTMNA2014, Tang, Tang2} and references therein. In the local case, the Schr\"odinger equation appears also as an approximation of the Maxwell equation. The fractional case has been also widely investigated in \cite{Bieganowski, doO, Binlin, He, Ros, Ambrosio, Davila, Fall, Davila2, Fall2, Frank2, Shang, Torres}; see also references therein.

There is a lot of results concerning the case $\Gamma = 0$ and $\mu = 0$. The existence of nontrivial solutions was obtained by S. Secchi in \cite{secchi} for a subcritical $f \in \cC^1 (\R^N \times \R)$ satisfying the Ambrosetti-Rabinowitz type condition $0 < \mu F(x,u) < u f(x,u)$ for $\mu > 2$ and a coercive potential $V \in \cC^1 (\R^N)$. The Nehari manifold method was also introduced in \cite{secchi} with the classical monotonicity condition: $t \mapsto t^{-1} u f(x, tu)$ is increasing on $(0, \infty)$. The variational setting for fractional equations was also provided in \cite{ServadeiValdinoci}. M. Cheng proved in \cite{Cheng} that (\ref{eq}) has a nontrivial solution for the subcritical nonlinearity $f(x,u) = |u|^{p-1}u + \omega u$ and a coercive potential $V(x) > 1$ for a.e. $x \in \R^N$. He showed also that there is a ground state solution being minimizer on the Nehari manifold for $0 < \omega < \lambda$, where $\lambda = \inf \sigma (\mathcal{A})$ and $\sigma (\mathcal{A})$ is the spectrum of the self-adjoint operator $\mathcal{A} := (-\Delta)^{\alpha / 2} + V(x)$ on $L^2 (\R^N)$. For $f(x,u) = |u|^{p-1}$, where $p$ is a subcritical exponent, there is a positive and spherically symmetric solution (see \cite{Dipierro}). The uniqueness of ground states $Q=Q(|x|)\geq 0$ of an equation $(-\Delta)^{\alpha / 2} Q + Q - Q^{\beta + 1} = 0$ in $\R$ was obtained by R.L. Frank and E. Lenzmann in \cite{Frank}. Recently, S. Secchi proved the existence of radially symmetric solution of $(-\Delta)^{\alpha / 2} u + V(x) u = g(u)$ for $g$ which does not satisfy the Ambrosetti-Rabinowitz condition (\cite{secchiTMNA}). Such a result was known before for $\alpha = 2$ and constant potentials $V$ (\cite{BerestyckiLions}). 

The local case $\alpha = 2$ with $V_{loc} = 0$, $\Gamma = 0$ and $q=2$ has been studied by Q. Guo and J. Mederski (see \cite{GuoMederski}) using Cerami sequences under more general assumption on $V=V_{per}$ that zero lies in the spectral gap of $-\Delta + V(x)$. 

In our case the nonlinear term depends on $x$, does not satisfy the Ambrosetti-Rabinowitz condition and the classical monotonicity condition is violated, moreover the potential is singular and the term $-\frac{\mu}{|x|^\alpha}$ does not belong to the Kato's class. Our main tool, which allows us to deal with such potentials, is the fractional Hardy inequality. This fact has been intensively studied by many authors (see for instance \cite{BogdanDyda, Dyda, DydaV, FrankSeiringer, FrankSeiringer2}). To deal with the sign-changing behaviour we use the variational setting based on the Nehari manifold technique from \cite{BieganowskiMederski}.

Now we state our main results concerning the existence of ground state solutions depending on the sign of $V_{loc}$ and the sign of $\mu < \mu^*$, where $\mu^*$ is a constant defined below.

\begin{Th}\label{ThMain1}
Suppose that ($V$), ($K$) (F1)--(F4) are satisfied, $V_{loc} \equiv 0$ or $V_{loc} < 0$, and $0 \leq \mu < \mu^*$. Then \eqref{eq} has a ground state, i.e. there is a nontrivial critical point $u$ of $\J$ such that $\J(u)=\inf_{\cN}\J$. 
\end{Th}

\begin{Th}\label{ThMain2}
Suppose that ($V$), ($K$), (F1)--(F4) are satisfied, $\mu < 0$ and
$$
V_{loc} (x) > \frac{\mu}{|x|^\alpha} \ \mbox{for} \ \mbox{a.e.} \ x \in \R^N \setminus \{0\},
$$
in particular $V_{loc}> 0$ can be considered. Then \eqref{eq} has no ground states.
\end{Th}

We are also interestend in the behaviour of solutions in asymptotic cases.

\begin{Th}\label{ThAsymptoticGamma}
Suppose that ($V$), (F1)--(F5) are satisfied, $V_{loc} \equiv 0$ and $\mu=0$. Moreover assume that every function in the sequence $(K_n)$ satisfies $(K)$ and $K_n \to 0$ in $L^\infty (\R^N)$. If $u_n$ is a ground state solution of \eqref{eq} with $K \equiv K_n$, then there is a sequence $(z_n) \subset \mathbb{Z}^N$ such that
$$
u_n (\cdot - z_n) \to u_0 \quad \mathrm{in} \ H^{\alpha / 2} (\R^N),
$$
where $u_0$ is a ground state solution of \eqref{eq} with $K \equiv 0$.
\end{Th}

\begin{Th}\label{ThAsymptoticMu}
Suppose that ($V$), ($K$), (F1)--(F5) are satisfied, $V_{loc} \equiv 0$. Moreover assume that the sequence $(\mu_n)$ satisfies $0 \leq \mu_n < \mu^*$ and $\mu_n \to 0$. If $u_n$ is a ground state solution of \eqref{eq} with $\mu = \mu_n$, then there is a sequence $(z_n) \subset \mathbb{Z}^N$ such that
$$
u_n (\cdot - z_n) \to u_0 \quad \mathrm{in} \ H^{\alpha / 2} (\R^N),
$$
where $u_0$ is a ground state solution of \eqref{eq} with $\mu = 0$.
\end{Th}

To our best knowledge, Theorems \ref{ThMain1} and \ref{ThMain2} were known before only for $\alpha = 2$, $V_{loc} \equiv 0$, $q=2$ and $K \equiv 0$ (see \cite{GuoMederski}). Similarly as in Section \ref{sect:AsymptoticGamma} we can show Theorem \ref{ThAsymptoticGamma} in the local case ($\alpha = 2$), which is a new result. Theorem \ref{ThAsymptoticMu} has been shown before only for $\alpha = 2$, $q=2$ and $K \equiv 0$ (see \cite{GuoMederski}).

Let us briefly describe the structure of this paper. In the second section we introduce some preliminary facts - the fractional Hardy inequality and the variational setting, which allows us to deal with sign-changing nonlinearities. The third section concerns the analysis of bounded Palais-Smale sequences. In the fourth section we provide the proof of Theorem \ref{ThMain1} using preliminary facts and the decomposition of Palais-Smale sequences. The fifth section contain the proof of Theorem \ref{ThMain2}. In sixth and seventh sections we provide proofs of Theorems \ref{ThAsymptoticGamma} and \ref{ThAsymptoticMu}.

\section{Preliminary facts}
\label{sect:Preliminary}

Let us recall the fractional Hardy inequality, which is the main tool and allows us to deal with Hardy-type potentials.

\begin{Lem}[{\cite{FrankSeiringer}[Theorem 1.1]}] \label{Lem:HardyIneq}
There is $H_{N,\alpha} > 0$ such that for every $u \in H^{\alpha / 2} (\R^N)$ and $N > \alpha$ there holds
$$
\iint_{\R^N \times \R^N} \frac{|u(x)-u(y)|^2}{|x-y|^{N+\alpha}} \, dx \, dy \geq H_{N,\alpha} \int_{\R^N} \frac{|u(x)|^2}{|x|^{\alpha}} \, dx.
$$
\end{Lem}
\noindent Moreover, in view of \cite{FrankSeiringer}, the sharp constant $H_{N,\alpha}$ can be estabilished and it is equal
$$
H_{N,\alpha} = 2 \pi^{N/2} \frac{\Ga \left( \frac{N+\alpha}{4} \right)^2  | \Ga ( - \alpha / 2)|}{\Ga \left( \frac{N-\alpha}{4} \right)^2 \Ga \left( \frac{N+\alpha}{2}  \right)}.
$$
Define 
$$
\mu^* := H_{N,\alpha} c_{N,\alpha} = 2^{\alpha} \left( \frac{\Ga \left( \frac{N+\alpha}{4} \right) }{\Ga \left( \frac{N-\alpha}{4} \right) } \right)^2.
$$
Note that in the local case ($\alpha = 2$) we obtain
$$
\mu^* = 4 \left( \frac{\Gamma \left(\frac{N}{4} - \frac{1}{2}+1\right)}{\Gamma \left(\frac{N}{4} - \frac{1}{2}\right)} \right)^2 = 4 \left( \frac{N}{4} - \frac{1}{2} \right)^2 = \frac{(N-2)^2}{4}.
$$
This same constant has been obtained in \cite{GuoMederski} and it is the sharp constant in the Hardy inequality in $H^1 (\R^N)$. 

\begin{Lem}\label{norm-eqv}
Let $0 \leq \mu < \mu^*$. There exists $1 > D_{N,\alpha,\mu} > 0$ such that for any $u \in H^{\alpha / 2} (\R^N)$ 
$$
D_{N,\alpha,\mu} \|u\|^2 \leq \|u\|^2 - \mu \int_{\R^N} \frac{u^2}{|x|^\alpha} \, dx \leq \|u\|^2.
$$
\end{Lem}

\begin{proof}
We have
$$
\| u \|^2 = c_{N,\alpha} \iint_{\R^N \times \R^N} \frac{|u(x)-u(y)|^2}{|x-y|^{N+\alpha}} \, dx \, dy + \int_{\R^N} V(x) |u|^2 \, dx.
$$
Put
$$
\| u \|^2_\mu := \| u \|^2 - \mu \int_{\R^N} \frac{|u|^2}{|x|^\alpha} \, dx.
$$
In view of Lemma \ref{Lem:HardyIneq}
$$
c_{N,\alpha} \iint_{\R^N \times \R^N} \frac{|u(x)-u(y)|^2}{|x-y|^{N+\alpha}} \, dx \, dy \geq c_{N,\alpha} H_{N,\alpha} \int_{\R^N} \frac{|u(x)|^2}{|x|^\alpha} \, dx.
$$
 Then $\mu < \mu^*$ means that
$$
1 \geq 1 - \frac{\mu}{c_{N,\alpha} H_{N,\alpha}} > 0.
$$
Then
\begin{align} \label{Eq:2.1}
& c_{N,\alpha} \iint_{\R^N \times \R^N} \frac{|u(x)-u(y)|^2}{|x-y|^{N+\alpha}} \, dx \, dy - \mu \int_{\R^N} \frac{|u(x)|^2}{|x|^\alpha} \, dx \\ \geq& \left(1 - \frac{ \mu}{c_{N,\alpha} H_{N,\alpha} } \right) c_{N,\alpha} \iint_{\R^N \times \R^N}  \frac{|u(x)-u(y)|^2}{|x-y|^{N+\alpha}} \, dx \, dy \geq 0.\nonumber
\end{align}
On the other hand
\begin{equation} \label{Eq:2.2}
\| u\|_\mu^2 \geq \left( 1 - \frac{ \mu}{c_{N,\alpha} H_{N,\alpha} } \right) c_{N,\alpha} \iint_{\R^N \times \R^N}  \frac{|u(x)-u(y)|^2}{|x-y|^{N+\alpha}} \, dx \, dy.
\end{equation}
From (\ref{Eq:2.1}) we have
\begin{equation} \label{Eq:2.3}
\| u \|_\mu^2 \geq \int_{\R^N} V(x) u^2 \, dx.
\end{equation}
Combining (\ref{Eq:2.2}) and (\ref{Eq:2.3}) we obtain
\begin{align*}
\| u \|_\mu^2 &= \frac12 \|u\|_\mu^2 + \frac12 \|u\|_\mu^2 \\ &\geq \frac12 \left( 1 - \frac{ \mu}{c_{N,\alpha} H_{N,\alpha} } \right) c_{N,\alpha} \iint_{\R^N \times \R^N} \frac{|u(x)-u(y)|^2}{|x-y|^{N+\alpha}} \, dx \, dy + \frac12 \int_{\R^N} V(x) u^2 \, dx \\ &\geq \frac12 \left( 1 - \frac{\mu}{ c_{N,\alpha}H_{N,\alpha} } \right) \|u\|^2.
\end{align*}
\end{proof}

We will briefly introduce the abstract setting from \cite{BieganowskiMederski}. Suppose that $E$ is a Hilbert space with respect to the norm $\| \cdot \|_E$. Let us consider a functional $\cJ : E \rightarrow \R$ of the general form
$$
\cJ(u) = \frac{1}{2} \|u\|^2_E - \cI (u),
$$
where $\cI : E \rightarrow \R$ is of $\cC^1$-class. The Nehari manifold for the functional $\cJ$ is given by
$$
\cN = \{ u \in E \setminus \{0\} \ : \ \cJ'(u)(u) = 0 \}.
$$ 
Let us recall a critical point theorem from \cite{BieganowskiMederski}, which is based on the approach of \cite{SzulkinWeth}, \cite{MederskiARMA} and \cite{BartschMederski1}.

\begin{Th}[{\cite{BieganowskiMederski}[Theorem 2.1]}]\label{ThSetting}
Suppose that the following conditions hold:
\begin{enumerate}
\item[(J1)] there is $r>0$ such that $a:= \inf_{\|u\|_E=r} \cJ (u) > \cJ(0) = 0$;
\item[(J2)] there is $q \geq 2$ such that $\cI (t_n u_n) / t_n^q \to \infty$ for any $t_n \to\infty$ and $u_n\to u \neq 0$ as $n\to\infty$;
\item[(J3)] for $t \in (0,\infty) \setminus \{1\}$ and $u \in \cN$
$$
\frac{t^2-1}{2} \cI'(u)(u) - \cI(tu)+\cI(u) < 0;
$$
\item[(J4)] $\cJ$ is coercive on $\cN$.
\end{enumerate}
Then $\inf_\cN \cJ > 0$ and there exists a bounded minimizing sequence for $\cJ$ on $\cN$, i.e. there is a sequence $(u_n) \subset \cN$ such that $\cJ(u_n) \to \inf_\cN \cJ$ and $\cJ'(u_n) \to 0$.
\end{Th} 

From the proof of \cite{BieganowskiMederski}[Theorem 2.1] it follows that for every $u \in E \setminus \{ 0\}$ there is unique number $t(u) > 0$ such that $t(u)u \in \cN$. Moreover the function $m : \{ u \in E \ : \ \|u\|_E = 1 \} \rightarrow \cN$ given by $m(u) = t(u)u$ is an homeomorphism. We will check that (J1)--(J4) are satisfied for $(E, \| \cdot \|_E) = \left( H^{\alpha / 2} (\R^N), \| \cdot \|_\mu \right)$, provided that $0 \leq \mu < \mu^*$. 

\begin{Lem}
Let $0 \leq \mu < \mu^*$ and assume that (F1)--(F4), ($K$) and ($V$) are satisfied. (J1)--(J4) are satisfied for $(E, \| \cdot \|_E) = \left( H^{\alpha / 2} (\R^N), \| \cdot \|_\mu \right)$.
\end{Lem}

\begin{proof}
\begin{enumerate}
	\item[(J1)] Fix $\varepsilon > 0$. Observe that (F1) and (F2) implies that $F(x,u) \leq \varepsilon |u|^2 + C_\varepsilon |u|^p$ for some $C_\varepsilon > 0$. Therefore 
	$$
	\int_{\R^N} F(x,u) \, dx - \int_{\R^N} \frac{1}{q} \Ga (x) |u|^q \, dx \leq \int_{\R^N} F(x,u) \, dx \leq C_\mu (\varepsilon \|u\|_\mu^2 + C_\varepsilon \|u\|_\mu^p),
	$$
for some constant $C_\mu > 0$ provided by the Sobolev embedding theorem and Lemma \ref{norm-eqv}. Thus there is $r_\mu > 0$ such that 
$$
\int_{\R^N} F(x,u) \, dx - \int_{\R^N} \frac{1}{q} \Ga (x) |u|^q \, dx \leq \frac{1}{4} \|u\|_\mu^2
$$
for $\|u\|_\mu \leq r_\mu$.
Therefore
$$
\cJ (u) \geq \frac{1}{4} \|u\|_\mu^2 = \frac{1}{4} r_\mu^2 > 0
$$
for $\|u\|_\mu = r_\mu$.
	\item[(J2)] By (F3) and Fatou's lemma we get
	$$
	\cI (t_n u_n) / t_n^q = \int_{\R^N} \frac{F(x, t_n u_n)}{t_n^q} \, dx - \frac{1}{q} \int_{\R^N} \Ga (x) |u_n|^q \, dx \to \infty.
	$$
	\item[(J3)] Fix $u \in \cN$ and consider
$$
\psi (t) = \frac{t^2-1}{2} \cI'(u)(u) - \cI(tu) + \cI(u)
$$
for $t \geq 0$. Then $\psi(1) = 0$ and
$$
\frac{d\psi(t)}{dt} = \int_{\R^N} tf(x,u)u - f(x,tu)u \, dx + (t^{q-1} - t) \int_{\R^N} \Ga(x) |u|^q \, dx.
$$
Since $u \in \cN$
$$
\int_{\R^N} f(x,u)u \, dx - \int_{\R^N} \Ga (x) |u|^q \, dx = \|u\|_\mu^2 > 0.
$$
Therefore, for $t>1$, we have
$$
\frac{d\psi(t)}{dt} < \int_{\R^N} t^{q-1} f(x,u)u - f(x,tu)u \, dx = t^{q-1} \int_{\R^N} f(x,u)u - \frac{f(x,tu)u}{t^{q-1}} \, dx < 0,
$$
by (F4). Similarly $\frac{d\psi(t)}{dt} > 0$ for $t < 1$. Therefore $\psi(t) < \psi(1) = 0$ for $t \neq 1$, i.e.
$$
\frac{t^2-1}{2} \cI'(u)(u) - \cI(tu) + \cI(u) < 0.
$$

\item[(J4)] Let $(u_n) \subset \cN$ be a sequence such that $\|u_n\|_\mu \to \infty$ as $n\to\infty$. (F3) implies that
$$
f(x,u)u = q\int_0^u \frac{f(x,u)}{u^{q-1}}s^{q-1}\,ds \geq q\int_0^u \frac{f(x,s)}{s^{q-1}}s^{q-1}\,ds=q F(x,u)
$$
for $u\geq 0$ and similarly $f(x,u)u\geq q F(x,u)$ for $u<0$.
Therefore
\begin{align*}
\cJ(u_n) &= \frac{1}{2} \|u_n\|_\mu^2 - \int_{\R^N} F(x, u_n) \, dx + \frac{1}{q} \int_{\R^N} \Ga (x) |u_n|^q \, dx = \\
&= \left( \frac{1}{2} - \frac{1}{q} \right) \|u_n\|_\mu^2 + \int_{\R^N} \frac{1}{q} f(x, u_n) u_n - F(x, u_n) \, dx \geq \\
&\geq \left( \frac{1}{2} - \frac{1}{q} \right) \|u_n\|_\mu^2 \to \infty
\end{align*}
as $n\to\infty$, since $q > 2$.
\end{enumerate}
\end{proof}

The following fact is very useful to deal with the Hardy-type term and plays a very important role in the proof of the decomposition result.

\begin{Lem}\label{hardyLemma}
If $|x_n|\to\infty$, then for any $u \in H^{\alpha / 2}(\R^N)$,
$$
\int_{\R^N} \frac{1}{|x|^\alpha} |u(\cdot - x_n)|^2 \, dx \to 0.
$$
\end{Lem}

\begin{proof}
Let $\varphi_m \in \cC_0^\infty (\R^N)$ and $\varphi_m \to u$ in $H^{\alpha / 2} (\R^N)$ as $m \to \infty$. Take $R_m > 0$ large enough that
$$
\supp \varphi_m \subset B(0, R_m).
$$
Obviously, for any $m$ there is $n(m)$ such that $|x_{n(m)}|-R_m \geq m$ and $n(m)$ is increasing. We get
\begin{eqnarray*}
\int_{\R^N} \frac{1}{|x|^\alpha} | \varphi_m (\cdot - x_n)|^2 \, dx &=& \int_{\R^N} \frac{1}{|x+x_n|^\alpha} |\varphi_m|^2 \, dx = \int_{B(0,R_m)} \frac{1}{|x+x_n|^\alpha} |\varphi_m|^2 \, dx \\
&\leq& \frac{1}{(|x_n|-R_m)^\alpha} \int_{B(0,R_m)} |\varphi_m|^2 \, dx \leq \frac{1}{m^\alpha} \int_{\R^N} |\varphi_m|^2 \, dx \to 0
\end{eqnarray*}
We have
\begin{eqnarray*}
\int_{\R^N} \frac{1}{|x|^\alpha} | u(\cdot - x_n) |^2 \, dx &\leq& \int_{\R^N} \frac{1}{|x|^\alpha} | u(\cdot - x_n) - \varphi_m (\cdot - x_n) |^2 \, dx + \int_{\R^N} \frac{1}{|x|^\alpha} | \varphi_m (\cdot - x_n)|^2 \, dx \\ &=& \int_{\R^N} \frac{1}{|x|^\alpha} | u(\cdot - x_n) - \varphi_m (\cdot - x_n) |^2 \, dx + o(1).
\end{eqnarray*}
In view of the fractional Hardy inequality we obtain
$$
\int_{\R^N} \frac{1}{|x|^\alpha} | u(\cdot - x_n) - \varphi_m (\cdot - x_n)|^2 \, dx \leq \frac{1}{H_{N,\alpha}} \iint_{\R^N \times \R^N} \frac{|u(x)- \varphi_m (x)- u(y) + \varphi_m(y)|^2}{|x-y|^{N+\alpha}} \, dx \, dy \to 0,
$$
since $\varphi_m \to u$ in $H^{\alpha / 2} (\R^N)$.
\end{proof}

\section{Profile decomposition of bounded Palais-Smale sequences}
\label{sect:Splitting}

The main theorem in this section is a modification of the decomposition results from \cite{BieganowskiMederski}[Theorem 4.1] and \cite{Bieganowski}[Theorem 3.1], in the spirit of \cite{JeanjeanTanaka}. We consider the functional $\cJ : H^{\alpha / 2} (\R^N) \rightarrow \R$ of the form
$$
\cJ (u) = \frac{1}{2} \|u\|^2 - \mu \int_{\R^N} \frac{u^2}{|x|^\alpha} \, dx - \int_{\R^N} G(x,u) \, dx,
$$
where the norm is defined by 
$$
\|u\|^2 = c_{N,\alpha} \iint_{\R^N \times \R^N} \frac{|u(x)-u(y)|^2}{|x-y|^{N+\alpha}} \, dx \, dy + \int_{\R^N} V(x) u^2 \, dx.
$$
The norm is associated with the following scalar product
$$
\langle u, v \rangle := c_{N,\alpha} \iint_{\R^N \times \R^N} \frac{|u(x)-v(y)|^2}{|x-y|^{N+\alpha}} \, dx \, dy + \int_{\R^N} V(x) u(x)v(x) \, dx.
$$
We suppose that $G(x,u) = \int_0^u g(x,s) \, ds$, where $g : \R^N \times \R \rightarrow \R$ satisfies:
\begin{itemize}
\item[($G1$)] $g(\cdot, u)$ is measurable and $\mathbb{Z}^N$-periodic in $x \in \R^N$, $g(x,\cdot)$ is continuous in $u \in \R$ for a.e. $x \in \R^N$;
\item[($G2$)] $g(x,u) = o(|u|)$ as $|u| \to 0^+$ uniformly in $x \in \R^N$;
\item[($G3$)] there exists $2 < r < 2_\alpha^*$ such that $\lim_{|u| \to \infty} g(x,u)/|u|^{r-1} = 0$ uniformly in $x \in \R^N$;
\item[($G4$)] for each $a<b$ there is a constant $c>0$ such that  $|g(x,u)|\leq c$ for a.e. $x\in\R^N$ and $a\leq u\leq b$.
\end{itemize} 
We will also denote
$$
\cJ_{\infty} (u) = \cJ (u) - \frac{1}{2} \int_{\R^N} V_{loc} (x) |u|^2 \, dx.
$$

\begin{Th}\label{ThDecomposition}
Suppose that ($G1$)--($G4$) and ($V$) hold and $0 \leq \mu < \mu^*$. Let $(u_n)$ be a bounded Palais-Smale sequence for $\cJ$. Then passing to a subsequence of $(u_n)$, there exist an integer $\ell > 0$ and sequences $(y_n^k) \subset \mathbb{Z}^N$, $w^k \in H^{\alpha / 2} (\R^N)$, $k = 1, \ldots, \ell$ such that:
\begin{itemize}
\item[$(a)$] $u_n \rightharpoonup u_0$ and $\cJ' (u_0) = 0$;
\item[$(b)$] $|y_n^k| \to \infty$ and $|y_n^k - y_n^{k'}| \to \infty$ for $k \neq k'$;
\item[$(c)$] $w^k \neq 0$ and $\cJ_{\infty}'(w^k) = 0$ for each $1 \leq k \leq \ell$;
\item[$(d)$] $u_n - u_0 - \sum_{k=1}^\ell w^k (\cdot - y_n^k) \to 0$ in $H^{\alpha / 2} (\R^N)$ as $n \to \infty$;
\item[$(e)$] $\cJ (u_n) \to \cJ(u_0) + \sum_{k=1}^\ell \cJ_{\infty} (w^k) + \frac{\mu}{2} \sum_{k=1}^\ell \int_{\R^N} \frac{|w^k|^2}{|x|^\alpha} \, dx$.
\end{itemize}
\end{Th}

\begin{Rem} \label{Rem:Epsilon}
Note that (G2)--(G4) imply that for every $\varepsilon > 0$ there is $C_\varepsilon > 0$ such that
$$
|g(x,u)| \leq \varepsilon |u| + C_\varepsilon |u|^{r-1}
$$
for any $u \in \R$ and a.e. $x \in \R^N$.
\end{Rem}

\begin{Rem}
In view of Remark \ref{Rem:Epsilon} and continuous Sobolev embeddings there is constant $\rho > 0$ such that
$$
\| v \| \geq \rho > 0
$$
for every nontrivial, critical point of $\cJ$ on $H^{\alpha / 2} (\R^N)$. It is also true for $\cJ_{\infty}$.
\end{Rem}

\begin{proof}
\textbf{Step 1:} \textit{We may find a subsequence of $(u_n)$ such that $u_n \rightharpoonup u_0$, where $u_0 \in H^{\alpha / 2} (\R^N)$ is a critical point of $\cJ$.} \\
Since $(u_n) \subset H^{\alpha / 2} (\R^N)$ is bounded, we may assume that (up to a subsequence) $u_n \rightharpoonup u_0$ in $H^{\alpha / 2} (\R^N)$ and $u_n(x) \to u_0(x)$ for a.e. $x \in \R^N$, for some $u_0 \in H^{\alpha / 2} (\R^N)$. Take any $\varphi \in \cC_0^\infty (\R^N)$ and observe that
$$
\cJ'(u_n)(\varphi) - \cJ'(u_0)(\varphi) = \langle u_n - u_0, \varphi \rangle - \mu \int_{\R^N} \frac{(u_n - u_0)\varphi}{|x|^\alpha} \, dx  - \int_{\R^N} \left( g(x,u_n) - g(x,u_0) \right) \varphi \, dx.
$$
By the weak convergence we have that $\langle u_n - u_0, \varphi \rangle \to 0$. Moreover, for any measurable set $E \subset \mathrm{supp}\, \varphi$ we have
$$
\int_E |g(x,u_n) \varphi| \, dx \leq C ( |u_n|_2 |\varphi \chi_E |_2 + |u_n|_r^{r-1} | \varphi \chi_E|_r )
$$
and therefore by the Vitali convergence theorem
$$
\int_{\mathrm{supp}\, \varphi} \left( g(x,u_n) - g(x,u_0) \right) \varphi \, dx \to 0.
$$
In view of the Hardy inequality (Lemma \ref{Lem:HardyIneq}), $(u_n)$ is bounded in $L^2 \left(\R^N, \frac{dx}{|x|^\alpha} \right)$, hence we may assume that
$$
u_n \weakto u_0 \quad \mathrm{in} \ L^2 \left(\R^N, \frac{dx}{|x|^\alpha} \right).
$$
Thus
$$
\mu \int_{\R^N} \frac{(u_n - u_0)\varphi}{|x|^\alpha} \, dx \to 0.
$$
Hence $\cJ'(u_n) (\varphi) \to \cJ'(u_0) (\varphi)$ and therefore $\cJ'(u_0) = 0$.

\textbf{Step 2:} \textit{Let $v_n^1 = u_n - u_0$. Suppose that}
\begin{equation}\label{eq:step2}
\sup_{z \in \R^N} \int_{B(z,1)} |v_n^1|^2 \, dx \to 0.
\end{equation}
\textit{Then $u_n \to u_0$ and (a)--(e) hold for $\ell = 0$.} \\
Let $v_n^1 = u_n - u_0$ and suppose that
$$
\sup_{z \in \R^N} \int_{B(z,1)} |v_n^1|^2 \, dx \to 0.
$$
Observe that
\begin{align*}
\cJ'(u_n)(v_n^1) &= \langle u_n, v_n^1 \rangle - \mu \int_{\R^N} \frac{u_n v_n^1}{|x|^\alpha} \, dx - \int_{\R^N} g(x,u_n) v_n^1 \, dx \\ &= \|v_n^1\|^2 + \langle u_0, v_n^1 \rangle - \mu \int_{\R^N} \frac{u_n v_n^1}{|x|^\alpha} \, dx - \int_{\R^N} g(x,u_n) v_n^1 \, dx
\end{align*}
and therefore
$$
\|v_n^1\|^2 = \cJ'(u_n) (v_n^1) + \int_{\R^N} \left( g(x, u_n) - g(x,u_0) \right) v_n^1 \, dx + \mu \int_{\R^N} \frac{(u_n - u_0) v_n^1}{|x|^\alpha} \, dx.
$$
We have $\cJ'(u_n) (v_n^1) \to 0$, since $u_n$ is a Palais-Smale sequence. Moreover, the Vitali convergence theorem and the Lion's lemma (\cite{secchi}[Lemma 2.4]) imply
$$
\int_{\R^N} \left( g(x, u_n) - g(x,u_0) \right) v_n^1 \, dx \to 0.
$$
Hence
$$
\|v_n^1\|^2 = \mu \int_{\R^N} \frac{ |v_n^1|^2}{|x|^\alpha} \, dx + o(1).
$$
Recall that, for $\mu < \mu^*$ we have
$$
D_{N,\alpha,\mu} \|v_n^1\|^2 \leq \|v_n^1\|^2 - \mu \int_{\R^N} \frac{|v_n^1|^2}{|x|^\alpha} \, dx = o(1).
$$
Therefore $u_n \to u_0$ in $H^{\alpha / 2}(\R^N)$ and by the continuity of $\cJ$ we have $\cJ(u_n) \to \cJ(u_0)$.

\textbf{Step 3:} \textit{Suppose that there is a sequence $(z_n) \subset \mathbb{Z}^N$ such that
$$
\liminf_{n\to\infty} \int_{B(z_n, 1+\sqrt{N})} |v_n^1|^2 \, dx > 0.
$$
Then there is $w \in H^{\alpha / 2} (\R^N)$ such that (up to a subsequence):}
$$
(i) \ |z_n| \to \infty, \quad (ii) \ u_n(\cdot + z_n) \rightharpoonup w \neq 0, \quad (iii) \ \cJ_{\infty}' (w) = 0.
$$
(i) and (ii) are standard. Put $v_n = u_n (\cdot + z_n)$ and denote that
\begin{eqnarray*}
& & \cJ_{\infty}'(v_n)(\varphi) - \cJ_{\infty}' (w)(\varphi) = \\ &=& \langle v_n-w, \varphi \rangle - \int_{\R^N} \left( g(x,v_n) - g(x,w) \right) \varphi \, dx - \mu \int_{\R^N} \frac{(v_n-w)\varphi}{|x|^\alpha} \, dx - \int_{\R^N} V_{loc} (x) \left( v_n - w \right) \varphi \, dx.
\end{eqnarray*}
By the weak convergence $\langle v_n-w, \varphi \rangle \to 0$. By the Vitali convergence theorem
$$
\int_{\R^N} \left( g(x,v_n) - g(x,w) \right) \varphi \, dx \to 0.
$$
Take a measurable set $E \subset \mathrm{supp}\, \varphi$ and observe that
$$
\int_{\R^N} |V_{loc} (x) \left( v_n - w \right) \varphi | \chi_E \, dx \leq |V_{loc}|_\infty \int_{\mathrm{supp}\, \varphi} |v_n - w| |\varphi \chi_E| \, dx \leq |V_{loc}|_\infty |v_n - w|_2 |\varphi \chi_E|_2
$$
and therefore by the Vitali convergence theorem
$$
\int_{\R^N} V_{loc} (x) \left( v_n - w \right) \varphi \, dx \to 0.
$$
Hence
$$
\cJ_{\infty}'(v_n)(\varphi) - \cJ_{\infty}' (w)(\varphi) = - \mu \int_{\R^N} \frac{(v_n-w)\varphi}{|x|^\alpha} \, dx + o(1).
$$
Again, in view of the Hardy inequality (Lemma \ref{Lem:HardyIneq}), we may assume that $v_n \weakto w$ in $L^2 \left(\R^N, \frac{dx}{|x|^\alpha} \right)$, thus
$$
\left| \int_{\R^N} \frac{(v_n-w)\varphi}{|x|^\alpha} \, dx \right| \to 0.
$$
Thus $\cJ_{\infty}' (v_n) (\varphi) \to \cJ_{\infty}' (w) (\varphi)$. We need to show that $\cJ_{\infty}' (v_n) (\varphi) \to 0$. In this purpose observe that
$$
\cJ' (u_n) \left( \varphi (\cdot - z_n) \right) \to 0,
$$
since $(u_n)$ is a Palais-Smale sequence. On the other hand
\begin{eqnarray*}
0 &\leftarrow& \cJ' (u_n) \left( \varphi (\cdot - z_n) \right) \\
&=& \langle v_n, \varphi \rangle - \int_{\R^N} g(x, v_n) \varphi \, dx - \int_{\R^N} V_{loc} (x) v_n \varphi \, dx + \int_{\R^N} V_{loc} (x) u_n \varphi (\cdot - z_n) \, dx - \mu \int_{\R^N} \frac{u_n \varphi(\cdot - z_n)}{|x|^\alpha} \, dx \\
&=& \cJ'_{\infty}(v_n) (\varphi) + \int_{\R^N} V_{loc} (x+z_n) v_n \varphi \, dx - \mu \int_{\R^N} \frac{u_n \varphi(\cdot - z_n)}{|x|^\alpha} \, dx.
\end{eqnarray*}
By the Vitali convergence theorem we have
$$
\int_{\R^N} V_{loc} (x+z_n) v_n \varphi \, dx \to 0.
$$
Moreover, in view of the H\"older inequality, boundedness of $(u_n)$ in $L^2 \left(\R^N, \frac{dx}{|x|^\alpha} \right)$ and Lemma \ref{hardyLemma} we have
$$
\left| \int_{\R^N} \frac{u_n \varphi(\cdot - z_n)}{|x|^\alpha} \, dx \right| \leq \left( \int_{\R^N} \frac{u_n^2}{|x|^\alpha} \, dx \right)^{\frac{1}{2}} \left( \int_{\R^N} \frac{|\varphi(\cdot - z_n)|^2}{|x|^\alpha} \, dx \right)^{\frac{1}{2}} \to 0.
$$
Finally $\cJ'_{\infty}(v_n) (\varphi) \to 0$. Thus $\cJ'_{\infty}(w) = 0$.

\textbf{Step 4:} \textit{Suppose that there exist $m \geq 1$, $(y_n^k) \subset \mathbb{Z}^N$, $w^k \in H^{\alpha / 2} (\R^N)$ for $1\leq k\leq m$ such that
\begin{eqnarray*}
|y_n^k| \to \infty, \ |y_n^k - y_n^{k'}| \to \infty \quad for \ k \neq k', \\
u_n (\cdot + y_n^k) \to w^k \neq 0, \quad for \ each \ 1\leq k \leq m, \\
\cJ_{\infty}'(w^k)=0, \quad for \ each \ 1\leq k \leq m.
\end{eqnarray*}
Then,
\begin{itemize}
\item[(1)] if $\sup_{z \in \R^N} \int_{B(z,1)} \left| u_n - u_0 - \sum_{k=1}^m w^k (\cdot - y_n^k) \right|^2 \, dx \to 0$ as $n\to\infty$, then
$$
\left\| u_n - u_0 - \sum_{k=1}^m w^k (\cdot - y_n^k) \right\| \to 0;
$$
\item[(2)] if there is $(z_n) \subset \mathbb{Z}^N$ such that 
$$
\liminf_{n \to \infty} \int_{B(z_n, 1+\sqrt{N})} \left| u_n - u_0 - \sum_{k=1}^m w^k (\cdot - y_n^k) \right|^2 \, dx  > 0,
$$
then there is $w^{m+1} \in H^{\alpha / 2} (\R^N)$ such that (up to subsequences):
\begin{itemize}
\item[(i)] $|z_n| \to \infty$, \ $|z_n - y_n^k| \to \infty$, \ for \ $1 \leq k \leq m$,
\item[(ii)] $u_n(\cdot + z_n) \rightharpoonup w^{m+1} \neq 0,$
\item[(iii)] $\cJ_{\infty}'(w^{m+1}) = 0.$
\end{itemize}
\end{itemize}
}
\noindent Put $\xi_n =  u_n - u_0 - \sum_{k=1}^m w^k (\cdot - y_n^k)$. 
\begin{enumerate}
\item[(1)] In view of Lion's lemma (\cite{secchi}[Lemma 2.4]) $\xi_n \to 0$ in $L^r (\R^N)$. Note that
$$
\cJ ' (u_n) (\xi_n) = \| \xi_n\|^2 + \langle u_0, \xi_n \rangle + \sum_{k=1}^m \langle w^k (\cdot - y_n^k), \xi_n \rangle - \mu \int_{\R^N} \frac{u_n \xi_n}{|x|^\alpha} \, dx - \int_{\R^N} g(x,u_n) \xi_n \, dx
$$
and therefore
\begin{eqnarray*}
\| \xi_n \|^2 = \cJ'(u_n) (\xi_n) - \sum_{k=1}^m \langle w^k (\cdot - y_n^k), \xi_n \rangle + \int_{\R^N} ( g(x,u_n) - g(x,u_0) ) \xi_n \, dx + \mu \int_{\R^N} \frac{(u_n-u_0)\xi_n}{|x|^\alpha} \, dx.
\end{eqnarray*}
We have $\cJ'(u_n) (\xi_n) \to 0$, since $(u_n)$ is a Palais-Smale sequence. Since $\cJ_{\infty}'(w^k (\cdot - y_n^k)) = 0$ we have
$$
\langle w^k (\cdot - y_n^k), \xi_n \rangle = \int_{\R^N} g(x,w^k (\cdot - y_n^k)) \xi_n \, dx + \int_{\R^N} V_{loc} (x) w^k (\cdot - y_n^k) \xi_n \, dx + \mu \int_{\R^N} \frac{w^k (\cdot - y_n^k) \xi_n}{|x|^\alpha} \, dx
$$
and therefore
\begin{eqnarray*}
\| \xi_n \|^2 &=& o(1) - \sum_{k=1}^m \int_{\R^N} g(x,w^k) \xi_n (\cdot + y_n^k) \, dx - \sum_{k=1}^m \int_{\R^N} V_{loc}(x) w^k (\cdot - y_n^k) \xi_n \, dx \\ &+& \int_{\R^N} ( g(x,u_n) - g(x,u_0) ) \xi_n \, dx + \mu \int_{\R^N} \frac{(u_n-u_0)\xi_n}{|x|^\alpha} \, dx - \mu \sum_{k=1}^m \int_{\R^N} \frac{w^k (\cdot - y_n^k) \xi_n}{|x|^\alpha} \, dx.
\end{eqnarray*}
We can easily show that
$$
- \sum_{k=1}^m \int_{\R^N} g(x,w^k) \xi_n (\cdot + y_n^k) \, dx - \sum_{k=1}^m \int_{\R^N} V_{loc}(x) w^k (\cdot - y_n^k) \xi_n \, dx + \int_{\R^N} ( g(x,u_n) - g(x,u_0) ) \xi_n \, dx \to 0,
$$
hence
$$
\|\xi_n\|^2 = \mu \int_{\R^N} \frac{(u_n-u_0)\xi_n}{|x|^\alpha} \, dx - \mu \sum_{k=1}^m \int_{\R^N} \frac{w^k (\cdot - y_n^k) \xi_n}{|x|^\alpha} \, dx + o(1) = \mu \int_{\R^N} \frac{|\xi_n|^2}{|x|^\alpha} \, dx + o(1).
$$
Since $\mu < \mu^*$ we have
$$
D_{N,\alpha,\mu} \|\xi_n\|^2 \leq \|\xi_n\|^2 - \mu \int_{\R^N} \frac{|\xi_n|^2}{|x|^\alpha} \, dx = o(1)
$$
and therefore
$$
\xi_n \to 0 \quad \mathrm{in} \ H^{\alpha / 2} (\R^N).
$$
\item[(2)] Suppose that
$$
\liminf_{n\to\infty} \int_{B(z_n, 1+\sqrt{N})} \left| u_n - u_0 - \sum_{k=1}^m w^k (\cdot - y_n^k) \right|^2 \, dx > 0.
$$
Then (i) and (ii) hold as in Step 3. Put $v_n = u_n (\cdot + z_n)$. Then for $\varphi \in \cC_0^\infty (\R^N)$ we have
$$
\cJ_{\infty}'(v_n) (\varphi) - \cJ_{\infty}' (w^{m+1}) (\varphi) \to 0
$$
and $\cJ_{\infty}'(v_n)(\varphi) \to 0$ as in Step 3.
\end{enumerate}

\textbf{Step 5:} \textit{Conclusion.} \\
In view of Step 1, we know that $u_n \rightharpoonup u_0$ and $\cJ'(u_0) = 0$, which completes the proof of (a). If condition \eqref{eq:step2} from Step 2 holds, then $u_n \to u_0$ and theorem is true for $\ell = 0$. On the other hand, one has
$$
\liminf_{n\to\infty} \int_{B(y_n,1)} |v_n^1|^2 \, dx > 0
$$
for some $(y_n) \subset \R^N$. For each $y_n \in \R^N$ we may find $z_n \in \mathbb{Z}^N$ such that
$$
B(y_n, 1) \subset B(z_n, 1+\sqrt{N}).
$$
Then
$$
\liminf_{n\to\infty} \int_{B(z_n, 1+\sqrt{N})} |v_n^1|^2 \, dx \geq \liminf_{n\to\infty}\int_{B(y_n,1)} |v_n^1|^2 \, dx > 0.
$$
Therefore in view of Step 3 we find $w$ such that (i)--(iii) hold. Let $y_n^1 = z_n$ and $w^1 = w$. If (1) from Step 4 holds with $m=1$, then (b)--(d) are true. Otherwise (2) holds and we put $(y_n^2) = (z_n)$ and $w^2=w$. Then we iterate the Step 4. To complete the proof of (b)--(d) it is sufficient to show that this procedure will finish after a finite number of steps. Indeed, observe that
$$
\lim_{n\to\infty} \|u_n\|^2-\|u_0\|^2-\sum_{k=1}^m\|w^k\|^2 = \lim_{n\to\infty} \left\| u_n - u_0 - \sum_{k=1}^m w^k (\cdot - y_n^k) \right\|^2 \geq 0
$$
for each $m \geq 1$. Since $w^k$ are critical points of $\cJ_{\infty}$, there is $\rho_0 > 0$ such that $\|w^k\| \geq \rho_0 > 0$, so after a finite number of steps, say $\ell$ steps, condition (1) in Step 4 will hold.

\textbf{Step 6:} \textit{We will show that $(e)$ holds:}
$$
\cJ(u_n)\to \cJ(u_0) + \sum_{k=1}^\ell \cJ_{\infty} (w^k).
$$
Note that
\begin{eqnarray*}
\cJ(u_n) &=& \cJ (u_0) + \cJ (u_n - u_0) + \frac12 \langle u_n - u_0, u_0 \rangle + \frac12 \langle u_0, u_n-u_0 \rangle \\
&+& \int_{\R^N} \left[ G(x, u_n-u_0) + G(x,u_0) - G(x,u_n) \right] \, dx - \mu \int_{\R^N} \frac{(u_n-u_0)u_0}{|x|^\alpha} \, dx.
\end{eqnarray*}
Since $u_n \rightharpoonup u_0$ we have
$$
\frac12 \langle u_n - u_0, u_0 \rangle \to 0, \quad \frac12 \langle u_0, u_n-u_0 \rangle \to 0.
$$
Thus
$$
\cJ(u_n) = \cJ (u_0) + \cJ (u_n - u_0) + \int_{\R^N} \left[ G(x, u_n-u_0) + G(x,u_0) - G(x,u_n) \right] \, dx - \mu \int_{\R^N} \frac{(u_n-u_0)u_0}{|x|^\alpha} \, dx + o(1).
$$
Let us consider the function $H : \R^N \times [0,1] \rightarrow \R$ given by $H(x,t)=G(x,u_n-tu_0)$. Therefore
$$
G(x,u_n-u_0)-G(x,u_n) = H(x,1)-H(x,0) = \int_0^1 \frac{\partial H}{\partial s} (x,s) \, ds.
$$
Note that
\begin{eqnarray*}
\int_{\mathbb{R}^N} [ G(x,u_n-u_0) + G(x,u_0) - G(x,u_n)] \, dx &=& \int_{\R^N} \left[ \int_0^1 \frac{\partial H}{\partial s} (x,s) \, ds + G(x,u_0) \right] \, dx \\
&=& \int_{\R^N} \int_0^1 \frac{\partial H}{\partial s} (x,s) \, ds \, dx + \int_{\R^N}G(x,u_0)\, dx\\
&=& \int_0^1 \int_{\R^N} -g(x,u_n-su_0)u_0  \, dx \, ds + \int_{\R^N}G(x,u_0)\, dx.
\end{eqnarray*}
Let $E \subset \R^N$ be a measurable set. From the H\"older inequality we have
\begin{eqnarray*}
\int_{E} |g(x, u_n-su_0)u_0| \, dx &\leq& C \int_E |u_n-su_0| |u_0| \, dx + C \int_E |u_n-su_0|^{r-1} |u_0| \, dx \\
&\leq& C |(u_n-su_0)\chi_E|_2^2 |u_0\chi_E|_2^2 + C|(u_n - su_0)\chi_E|_r^{r-1} |u_0 \chi_E|_r.
\end{eqnarray*}
Then $(g(x, u_n-su_0)u_0)$ is uniformly integrable and by the Vitali convergence theorem we get
$$
\int_0^1 \int_{\R^N} -g(x,u_n-su_0)u_0  \, dx \, ds \to \int_0^1 \int_{\R^N} -g(x,u_0-su_0)u_0  \, dx \, ds.
$$
On the other hand we we have
\begin{eqnarray*}
\int_0^1 \int_{\R^N} -g(x,u_0-su_0)u_0  \, dx \, ds &=&  \int_{\R^N} \int_0^1 -g(x,u_0-su_0)u_0  \, ds \, dx \\
&=& \int_{\R^N} \int_0^1 \frac{\partial}{\partial s} \left[ G(x, u_0-su_0) \right] \, ds \, dx \\
&=& \int_{\R^N} G(x,0)-G(x,u_0) \, dx = \int_{\R^N} -G(x,u_0) \, dx.
\end{eqnarray*}
Finally
$$
\int_{\mathbb{R}^N} [ G(x,u_n-u_0) + G(x,u_0) - G(x,u_n)] \, dx \to \int_{\mathbb{R}^N} [ G(x,u_0) - G(x,u_0)] \, dx = 0.
$$
We obtained that
$$
\cJ(u_n) = \cJ (u_0) + \cJ_{\infty} (u_n - u_0) + \frac{1}{2} \int_{\R^N} V_{loc} (x) (u_n - u_0)^2 \, dx - \mu \int_{\R^N} \frac{(u_n-u_0)u_0}{|x|^\alpha} \, dx + o(1).
$$
For a measurable set $E \subset \mathbb{R}^N$ there holds
$$
\int_E |V_{loc}(x)| |u_n-u_0|^2 \, dx \leq |V_{loc} \chi_E|_{s} |u_n - u_0|_{\frac{2s}{s-1}}^2,
$$
where $s \geq \frac{N}{\alpha}$. While $\frac{2s}{s-1} \geq 2$ and $\frac{2s}{s-1} = \frac{2}{1-\frac{1}{s}} \leq \frac{2}{1-\frac{\alpha}{N}} = \frac{2N}{N-\alpha}$, we have that $(u_n - u_0)$ is bounded in $L^{\frac{2s}{s-1}} (\R^N)$ and in view of the Vitali convergence theorem we have
$$
\int_{\R^N} V_{loc} (x) (u_n - u_0)^2 \, dx \to 0.
$$
Hence
$$
\cJ(u_n) = \cJ (u_0) + \cJ_{\infty} (u_n - u_0) - \mu \int_{\R^N} \frac{(u_n-u_0)u_0}{|x|^\alpha} \, dx + o(1).
$$
Since $u_n \weakto u_0$ in $L^2 \left( \R^N, \frac{dx}{|x|^\alpha} \right)$, we have
$$
\int_{\R^N} \frac{(u_n-u_0)u_0}{|x|^\alpha} \, dx \to 0.
$$
Hence
$$
\cJ(u_n) = \cJ (u_0) + \cJ_{\infty} (u_n - u_0) + o(1).
$$
We obtain that $\cJ_{\infty} (u_n - u_0) = \sum_{k=1}^\ell \cJ_{\infty} (w^k (\cdot - y_n^k)) + o(1)$ in the same way. Thus
\begin{eqnarray*}
\cJ(u_n) &=& \cJ (u_0) + \sum_{k=1}^\ell \cJ_{\infty} (w^k (\cdot - y_n^k)) + o(1) \\
&=& \cJ (u_0) + \sum_{k=1}^\ell \cJ_{\infty} (w^k) - \frac{\mu}{2} \sum_{k=1}^\ell \int_{\R^N} \frac{|w^k (\cdot - y_n^k)|^2}{|x|^\alpha} \, dx + \frac{\mu}{2} \sum_{k=1}^\ell \int_{\R^N} \frac{|w^k |^2}{|x|^\alpha} \, dx + o(1).
\end{eqnarray*}
From Lemma \ref{hardyLemma} we have
$$
\frac{\mu}{2} \sum_{k=1}^\ell \int_{\R^N} \frac{|w^k (\cdot - y_n^k)|^2}{|x|^\alpha} \, dx  \to 0
$$
and therefore
$$
\cJ(u_n) = \cJ (u_0) + \sum_{k=1}^\ell \cJ_{\infty} (w^k) + \frac{\mu}{2} \sum_{k=1}^\ell \int_{\R^N} \frac{|w^k |^2}{|x|^\alpha} \, dx + o(1).
$$
\end{proof}

\section{Existence of solutions}

\begin{proof}[Proof of Theorem \ref{ThMain1}]
From Theorem \ref{ThSetting} there is a bounded minimizing sequence $(u_n) \subset \cN$ such that
$$
\cJ'(u_n) \to 0, \quad \cJ(u_n) \to c,
$$
where
$$
c = \inf_{\cN} \cJ > 0.
$$
Suppose that $V_{loc} \equiv 0$. Then $\cJ = \cJ_{\infty}$ and in view of Theorem \ref{ThDecomposition} we have
$$
c \leftarrow \cJ(u_n) \to \cJ (u_0) + \sum_{k=1}^\ell \cJ (w^k) + \frac{\mu}{2} \sum_{k=1}^\ell \int_{\R^N} \frac{|w^k |^2}{|x|^\alpha} \, dx \geq \cJ(u_0) + \ell c.
$$
If $u_0 \neq 0$ we obtain $c \geq (\ell + 1 ) c$ and $\ell = 0$, thus $u_0$ is a ground state solution. If $u_0 = 0$ we obtain $\cJ(u_0) = \cJ(0)=0$ and $c \geq \ell c$. Since $c > 0$, we have $\ell = 1$ and $w^k \neq 0$ is a ground state.

\noindent Suppose now that $V_{loc} < 0$. Denote $c_{\infty} = \inf_{\cN_{\infty}} \cJ_{\infty} > 0$. As in \cite{BieganowskiMederski} we can show that $c_{\infty} > c$. Indeed, take a critical point $u_\infty \neq 0$ of $\cJ_\infty$ such that $\cJ_\infty (u_\infty) = c_\infty$. Let $t > 0$ be such that $tu_\infty \in \cN$. While $V(x) < V_{loc} (x)$, we obtain
$$
c_\infty = \cJ_\infty (u_\infty) \geq \cJ_\infty(t u_\infty) > \cJ (t u_\infty) \geq c > 0.
$$
Then
$$
c \leftarrow \cJ(u_n) \to \cJ(u_0) + \sum_{k=1}^\ell \cJ_{\infty} (w^k) + \frac{\mu}{2} \sum_{k=1}^\ell \int_{\R^N} \frac{|w^k |^2}{|x|^\alpha} \, dx \geq \cJ(u_0) + \ell c_{\infty}.
$$
Since $c_{\infty} > c$, we have $\ell = 0$ and $u_0 \neq 0$ is a ground state solution.
\end{proof}

\section{Nonexistence of ground states}

\begin{proof}[Proof of Theorem \ref{ThMain2}]
Suppose that $u \in \cN$ is a ground state solution of \eqref{eq}. Denote by $\cJ_{per}$ the energy functional with $\mu = 0$ and $V_{loc} \equiv 0$, and let $\cN_{per}$ be the corresponding Nehari manifold. Let $t > 0$ be such that $tu \in \cN_0$. Then
$$
c_{per} := \inf_{\cN_{per}} \cJ_{per} \leq \cJ_{per} (t u) = \cJ (t u) - \frac{1}{2} \int_{\R^N} V_{loc}(x) |tu|^2 \, dx + \frac{\mu}{2} \int_{\R^N} \frac{|tu|^2}{|x|^\alpha} \, dx < \cJ(tu) \leq \cJ(u) =:c.
$$
Fix $z \in \mathbb{Z}^N$ and $u_{per} \in \cN_{per}$. Then there is $t(z) > 0$ such that $t(z) u_{per}(\cdot - z) \in \cN$. Observe that
\begin{align*}
\frac{1}{|t(z)|^{q-2}} \left( \|u_0\|^2 - \mu \int_{\R^N} \frac{|u_0 (x-z)|^2}{|x|^\alpha} \, dx \right) &= \frac{1}{|t(z)|^{q}} \int_{\R^N} f(x, t(z)u_0) t(z) u_0 \, dx - \frac{1}{|t(z)|^{q}} \int_{\R^N} K(x) |t(z)|^q |u_0|^q \, dx \\
&\geq \frac{1}{|t(z)|^q} \int_{\R^N} q F(x, t(z)u_0) \, dx -  \int_{\R^N} K(x) |u_0|^q \, dx \\
&= q \int_{\R^N} \frac{ F(x, t(z)u_0)}{|t(z)|^q} \, dx -  \int_{\R^N} K(x) |u_0|^q \, dx.
\end{align*}
The right hand side tends to $\infty$ as $t(z) \to \infty$, while the left hand side stays bounded. Hence $(t(z))$ is bounded if $|z| \to \infty$. Hence, take any sequence $(z_n) \subset \mathbb{Z}^N$ such that $|z_n| \to \infty$. We may assume that $t(z_n) \to t_0$ as $n \to \infty$ and $t_0 \geq 0$. Observe that, in view of Lemma \ref{hardyLemma},
\begin{align*}
& \cJ_{per} (u_{per}) = \cJ_{per} ( u_{per} (\cdot - z)) \geq \cJ_{per} (t(z) u_{per} (\cdot - z)) \\
&= \cJ (t(z) u_{per}(\cdot - z)) - \frac{|t(z)|^2}{2} \int_{\R^N} V_{loc} (x) |u_{per}(x - z)|^2 \, dx + \frac{\mu |t(z)|^2}{2} \int_{\R^N} \frac{|u_{per}(x - z)|^2}{|x|^\alpha} \, dx \\
&\geq c - \frac{|t(z)|^2}{2} \int_{\R^N} V_{loc} (x+z) |u_{per}|^2 \, dx + \frac{\mu |t(z)|^2}{2} \int_{\R^N} \frac{|u_{per}(x - z)|^2}{|x|^\alpha} \, dx \\
 &\geq c + o(1).
\end{align*}
Taking infimum over $u_{per} \in \cN_{per}$ we obtain $c_{per} < c \leq c_{per}$ - a contradiction.
\end{proof}

\section{Asymptotic behaviour of ground states as $K_n \to 0$}
\label{sect:AsymptoticGamma}

Let $(K_n)$ be a sequence of functions such that for every $K_n$ the condition ($K$) holds and $K_n \to 0$ in $L^\infty (\R^N)$. Denote by $\cJ_n$ the energy functional for $K \equiv K_n$.

\begin{Lem}\label{Lem51}
Suppose that (F1), (F2) and ($V$) hold. Then there is a positive radius $r > 0$ such that
$$
a := \inf_{n \geq 1} \inf_{\|u\|_\mu = r} \cJ_n (u) > 0. 
$$
\end{Lem}

\begin{proof}
Fix $\varepsilon > 0$. From \eqref{eps} we have
$$
F(x,u) \leq \varepsilon |u|^2 + C_\varepsilon |u|^p
$$
for some $C_\varepsilon > 0$. Therefore, in view of Sobolev inequalities,
$$
\int_{\R^N} F(x,u) \, dx - \int_{\R^N} K_n(x) |u|^q \, dx \leq \int_{\R^N} F(x,u) \, dx \leq C \left( \varepsilon \|u\|_\mu^2 + C_\varepsilon \|u\|_\mu^p \right)
$$
for some $C > 0$. Hence, there is $r > 0$ such that
$$
\int_{\R^N} F(x,u) \, dx - \int_{\R^N} K_n(x) |u|^q \, dx \leq \frac{1}{4} \| u\|_\mu^2
$$
for $\|u\|_\mu \leq r$. Hence
$$
\cJ_n (u) \geq \frac{r^2}{4}
$$
for $\|u\|_\mu = r$.
\end{proof}

Recall that for any $n \geq 1$ there is a ground state solution $u_n$ in the corresponding Nehari manifold $\cN_n$ (Theorem \ref{ThMain1}). By $\cJ_0$ and $\cN_0$ we denote the energy functional and the corresponding Nehari manifold for the problem with $K \equiv 0$. In view of Theorem \ref{ThMain1} there is a ground state also for $\cJ_0$.

\begin{Lem}\label{Lem52}
Assume that (F1)--(F4) and (V) hold. We have
$$
\lim_{n\to\infty} \inf_{\cN_n} \cJ_n = \inf_{\cN_0} \cJ_0.
$$
\end{Lem}

\begin{proof}
Let $c_n = \inf_{\cN_n} \cJ_n > 0$ and $c_0 = \inf_{\cN_0} \cJ_0 > 0$. From Theorem \ref{ThMain1} there are $u_n \in \cN_n$ and $u_0 \in \cN_0$ such that
$$
\cJ_n (u_n) = c_n, \quad \cJ_0 (u_0) = c_0.
$$
Take $t_n > 0$ such that $t_n u_n \in \cN_0$ and observe that 
$$
c_n = \cJ_n (u_n) \geq \cJ_n (t_n u_n) = \cJ_0 (t_n) + t_n^q \int_{\R^N} K_n(x) |u_n|^q \, dx \geq c_0 + t_n^q \int_{\R^N} K_n(x) |u_n|^q \, dx.
$$
Now, take $t_n ' > 0$ such that $t_n ' u_0 \in \cN_n$ and note that
$$
c_0 = \cJ_0(u_0) \geq \cJ_0 (t_n' u_0) = \cJ_n (t_n' u_0) - (t_n')^q \int_{\R^N} K_n (x) |u_0|^q \, dx \geq c_n - (t_n')^q \int_{\R^N} K_n (x) |u_0|^q \, dx.
$$
Hence
\begin{equation}\label{ineq}
c_0 \leq c_0 + t_n^q \int_{\R^N} K_n(x) |u_n|^q \, dx \leq c_n \leq c_0 + (t_n')^q \int_{\R^N} K_n (x) |u_0|^q \, dx.
\end{equation}
We will show that $(t_n ')$ is bounded. Suppose by contradiction that, up to a subsequence, $t_n' \to \infty$. Since $t_n ' u_0 \in \cN_n$ we have that
$$
\| t_n ' u_0 \|_\mu^2 - \int_{\R^N} f(x, t_n' u_0) t_n' u_0 \, dx + \int_{\R^N} K_n (x) |t_n' u_0|^q \, dx = 0.
$$
Thus
$$
0 = \frac{\|u_0\|^2}{|t_n'|^{q-2}} - \int_{\R^N} \frac{f(x, t_n' u_0)u_0}{|t_n'|^{q-1}} \, dx + \int_{\R^N} K_n (x) | u_0|^q \, dx = o(1) - \int_{\R^N} \frac{f(x, t_n' u_0)u_0}{|t_n'|^{q-1}} \, dx \to -\infty.
$$
This contradiction shows that $(t_n')$ is bounded and therefore in view of \eqref{ineq}
$$
c_n \to c_0.
$$
\end{proof}

\begin{Lem}\label{Lem53}
For every choice of ground states $u_n$ of $\cJ_n$, the sequence $(u_n)$ is bounded in $H^{\alpha / 2} (\R^N)$.
\end{Lem}

\begin{proof}
Suppose that $\|u_n\|_\mu \to \infty$. Then
\begin{align*}
c_0 &= \lim_{n \to \infty} \cJ_n (u_n) = \lim_{n \to \infty} \left( \cJ_n (u_n) - \frac{1}{q} \cJ_n ' (u_n)(u_n) \right) \\
&= \lim_{n\to\infty} \left[ \left( \frac{1}{2} - \frac{1}{q} \right) \|u_n\|_\mu^2 + \frac{1}{q} \int_{\R^N} f(x,u_n)u_n - q F(x,u_n) \, dx \right] \\
&\geq  \lim_{n\to\infty} \left[ \left( \frac{1}{2} - \frac{1}{q} \right) \|u_n\|_\mu^2 \right] \to \infty
\end{align*}
- a contradiction.
\end{proof}

\begin{proof}[Proof of Theorem \ref{ThAsymptoticGamma}]
Let $\mu = 0$ and $V_{loc} \equiv 0$. We claim that there is a sequence $(y_n) \subset \mathbb{Z}^N$ such that
$$
\liminf_{n\to\infty} \int_{B(y_n, 1+\sqrt{N})} |u_n|^2\, dx > 0.
$$
Indeed, suppose that it is not true. From the fractional Lion's lemma we have
$$
u_n \to 0 \ \mathrm{in} \ L^t (\R^N) \ \mathrm{for} \ 2 < t < 2^*_\alpha.
$$
Since $u_n \in \cN_n$, we have
\begin{align*}
\|u_n\|^2 &= \int_{\R^N} f(x,u_n)u_n \, dx - \int_{\R^N} K_n(x) |u_n|^q \, dx \\ &\leq \varepsilon |u_n|_2^2 + C_\varepsilon |u_n|_p^p - \int_{\R^N} K_n(x) |u_n|^q \, dx \\ &\to \varepsilon \limsup_{n\to\infty} |u_n|_2^2
\end{align*}
and taking $\varepsilon \to 0^+$ we obtain $\|u_n\| \to 0$, and therefore $u_n \to 0$ in $H^{\alpha / 2}$. However, in view of Lemma \ref{Lem51} we have
$$
\cJ_n (u_n) \geq \cJ_n \left( r \cdot \frac{u_n}{\|u_n\|} \right) \geq a > 0
$$
and on the other hand
\begin{align*}
\limsup_{n\to\infty} \cJ_n (u_n) = -\limsup_{n \to \infty} \int_{\R^N} F(x,u_n) \, dx \leq 0
\end{align*}
- a contradiction. Thus
$$
\liminf_{n\to\infty} \int_{B(y_n, 1+\sqrt{N})} |u_n|^2\, dx > 0
$$
for some $(y_n) \subset \mathbb{Z}^N$. Then, in view of Lemma \ref{Lem53}, (up to a subsequence) there is $u \neq 0$ such that
\begin{align*}
u_{n} (\cdot + y_{n}) \to u \quad &\mathrm{in} \ L^2_{loc} (\R^N), \\
u_{n} (\cdot + y_{n}) \weakto u \quad &\mathrm{in} \ H^{\alpha / 2} (\R^N), \\
u_{n} (x+y_{n}) \to u(x) \quad &\mathrm{for} \ \mathrm{a.e.} \ x \in \R^N.
\end{align*}
Denote $w_n := u_{n}(\cdot + y_{n})$. Take any $\psi \in \cC^\infty_0 (\R^N)$. We have
$$
\cJ_0' \left( w_{n} \right) (\psi) = \cJ_{n}' (u_{n}) \left( \psi (\cdot - y_{n}) \right) -  \int_{\R^N} K_n (x) |u_{n}|^{q-2} u_n \psi(\cdot - y_{n}) \, dx = -  \int_{\R^N} K_n (x) |u_{n}|^{q-2} u_n \psi(\cdot - y_{n}) \, dx.
$$
Moreover
$$
\left| \int_{\R^N} K_n (x) |u_{n}|^{q-2} u_{n} \psi(\cdot - y_{n}) \, dx \right| \leq | K_n |_\infty  \int_{\R^N} |u_{n}|^{q-1} | \psi(\cdot - y_{n}) | \, dx \leq |K_n|_\infty |w_{n}|_{q}^{q-1} | \psi |_{q}.
$$
Therefore
$$
| \cJ_0' \left( w_{n} \right) (\psi) | \leq |K_n|_\infty |w_{n}|_{q}^{q-1} | \psi |_{q}
$$
The sequence $(w_{n})$ is bounded in $H^{\alpha / 2} (\R^N)$ (see Lemma \ref{Lem53}), which implies boundedness in $L^q (\R^N)$. Hence $\cJ_0' \left( w_{n} \right) (\psi) \to 0$. 

\noindent On the other hand - we have
$$
\cJ_0' (w_{n})(\psi) = \langle w_n, \psi \rangle - \int_{\supp \psi} f(x,w_n) \psi \, dx + \int_{\supp \psi} K_n(x) |w_n|^{q-2} w_n \psi \, dx.
$$
In view of the weak convergence $w_n \rightharpoonup u$ we have $\langle w_n, \psi \rangle \to 0$. Take any measurable set $E \subset \supp \psi$ and note that
$$
\int_E | f(x, w_n) \psi | \, dx \leq C \int_E |w_n \psi| + |w_n^{p-1} \psi| \, dx \leq C \left( | w_n |_{2} |\psi \chi_E |_{2} + |w_n|_{p}^{p-1} | \psi \chi_E |_{p} \right).
$$
Hence, for every $\varepsilon > 0$ there is $\delta > 0$ such that for $|E| < \delta$
$$
\int_E | f(x, w_n) \psi | \, dx < \varepsilon.
$$
Thus, in view of the Vitali convergence theorem
$$
\int_{\supp \psi} f(x,w_n) \psi \, dx \to \int_{\supp \psi} f(x,u) \psi \, dx.
$$
Similarly
$$
\int_{\supp \psi} |w_n^{q-1}  \psi | \, dx \leq |w_n|_{q}^{q-1} | \psi \chi_E |_{q} < \varepsilon
$$
for sufficiently small $\delta > 0$ and $|E| < \delta$. Hence
$$
\int_{\supp \psi} |w_n^{q-2}| w_n  \psi  \, dx \to \int_{\supp \psi} |u^{q-2}| u  \psi  \, dx.
$$
Thus
$$
\cJ_0' (w_{n})(\psi) \to \cJ_0'(u)(\psi).
$$
Therefore $\cJ_0'(u)(\psi) = 0$ and $u$ is the critical point of $\cJ_0$. Put
$$
c_0 := \inf_{\cN_0} \cJ_0, \quad c_n := \inf_{\cN_{n}} \cJ_{n} = \cJ_{n} (u_n).
$$
In view of Lemma \ref{Lem52}, we have $c_n \to c_0$ as $n \to \infty$, i.e. $\cJ_{n} (u_n) \to c_0$. The Fatou's lemma gives
\begin{eqnarray} \label{c}
& & c_0 = \liminf_{n\to\infty} \cJ_{n} (u_{n}) = \liminf_{n\to\infty} \left( \cJ_{n} (u_{n}) - \frac{1}{2} \cJ_{n}' (u_{n})(u_{n}) \right) \\
\nonumber &=& \liminf_{n\to\infty} \left[ \frac{1}{2} \int_{\R^N}  f(x,w_{n})w_{n} - 2F(x, w_{n}) \, dx - \left( \frac{1}{2} - \frac{1}{q} \right)  \int_{\R^N} K_n (x) |u_n|^q \, dx \right] \\
\nonumber &\geq& \frac{1}{2} \liminf_{n\to\infty} \left[ \int_{\R^N}  f(x,w_{n})w_{n} - 2F(x, w_{n}) \, dx \right] + \liminf_{n\to\infty} \left[ - \left( \frac{1}{2} - \frac{1}{q} \right) \int_{\R^N} K_n (x) |u_n|^q \, dx \right] \\
\nonumber &=& \frac{1}{2} \liminf_{n\to\infty} \left[ \int_{\R^N}  f(x,w_{n})w_{n} - 2F(x, w_{n}) \, dx \right] \geq  \frac{1}{2} \int_{\R^N} [ f(x,u)u - 2 F(x,u) ] \, dx \\
\nonumber &=& \int_{\R^N} \left[ \frac{1}{2} f(x,u)u -  F(x,u) \right] \, dx + \frac{1}{2} \cJ_0 '(u)(u) = \cJ_0 (u) \geq c_0.
\end{eqnarray}
Thus $u$ is a ground state for $\cJ_0$, i.e. $\cJ_0(u) = c_0$.

\noindent Now we are going to show that $u_n \to w$ in $H^{\alpha / 2}(\R^N)$. We have
\begin{eqnarray*}
\| w_n - u \|^2 &=& \cJ_{n}'(u_n) \left[ w_n (\cdot - y_n) - u (\cdot - y_n) \right] - \langle u, w_n - u \rangle \\
&-& \int_{\R^N} f(x,w_n)(w_n-u) \, dx +  \int_{\R^N} K_n (x) |w_n|^{q-2} w_n [w_n-u] \, dx
\end{eqnarray*}
By the weak convergence, we have $\langle u, w_n - u \rangle \to 0$. Moreover
$$
\cJ_{n}'(u_n) \left[ w_n (\cdot - y_n) - u (\cdot - y_n) \right] = \cJ_{n}'(u_n) (u_n) - \cJ_{n}'(u_n) \left( u (\cdot - y_n) \right) = - \cJ_{n}'(u_n) \left( u (\cdot - y_n) \right)
$$
and
$$
\cJ_{n}'(u_n) \left( u (\cdot - y_n) \right)  = 0.
$$
Therefore
$$
\| w_n - u \|^2 = - \int_{\R^N} f(x,w_n)(w_n-u) \, dx +  \int_{\R^N} K_n(x) |w_n|^{q-2} w_n [w_n-u] \, dx + o(1).
$$
Since $c_n \to c_0$, we have that $c_n = \cJ_{n} (u_n)$ is bounded. Moreover $(w_n)$ is bounded. Put
$$
G(x,u) := \frac{1}{2} f(x,u)u - F(x,u) \geq 0.
$$
Note that 
\begin{eqnarray*}
\int_{\R^N} G(x,w_n) - G(x,w_n - u) \, dx &=& \int_{\R^N} \int_0^1 \frac{d}{dt} G(x, w_n-u+tu) \, dt \, dx \\
&=& \int_0^1 \int_{\R^N} g(x, w_n-u+tu)u \, dx \, dt,
\end{eqnarray*}
where $g(x,u) := \frac{\partial}{\partial u} G(x,u)$. Observe that
$$
g(x,u) = \frac{1}{2} f_u ' (x,u) u + \frac{1}{2} f(x,u) - f(x,u) = \frac{1}{2} f_u ' (x,u) u - \frac{1}{2} f(x,u).
$$
From (F5) we see that
\begin{align}\label{ep}
|f_u ' (x,w_n-u+tu) (w_n-u+tu) | \leq c( |w_n-u+tu| +  |w_n-u+tu|^{p-1} ).
\end{align}
Since $(w_n-u+tu)_n$ is bounded in $H^{\alpha / 2} (\R^N)$, taking (F1), \eqref{ep} and H\"older inequality into account we see that for every $\varepsilon > 0$ there is $\delta > 0$ such that
$$
\int_E |g(x,w_n-u+tu)u| \, dx < \varepsilon
$$
for every $n$ and every measurable subset $E \subset \R^N$ such that $|E| < \delta$. Therefore $(g(x,w_n-u+tu)u)_n$ is uniformly integrable. Moreover for any $\varepsilon > 0$ there is a measurable subset $E \subset \R^N$ of finite measure $|E| < \infty$, such that for any $n \geq 1$ 
$$
\int_{\R^N \setminus E} |g(x, w_n-u+tu)u| \, dx < \varepsilon.
$$
Thus $(g(x, w_n-u+tu)u)_n$ is tight on $\R^N$. Hence, in view of the Vitali convergence theorem
$$
\int_{\R^N} g(x, w_n-u+tu)u \, dx \to \int_{\R^N} g(x, tu)u \, dx.
$$
Hence
\begin{eqnarray*}
\int_{\R^N} G(x,w_n) - G(x,w_n - u) \, dx &\to& \int_0^1 \int_{\R^N} g(x, tu)u \, dx \, dt \\
&=& \int_{\R^N} \int_0^1 g(x, tu)u \, dt \, dx = \int_{\R^N} G(x,u) \, dx.
\end{eqnarray*}
Recall that, in view of (\ref{c})
$$
c_0 = \lim_{n\to\infty} \left[ \int_{\R^N} G(x,w_n) \, dx \right] =  \int_{\R^N} G(x,u) \, dx.
$$
Hence
\begin{eqnarray*}
\int_{\R^N} G(x, w_n -u) \, dx \to 0.
\end{eqnarray*}
By (F5) we have 
$$
| w_n - u |_{r}^r = \int_{\R^N} |w_n - u|^r \, dx \leq \frac{2}{b} \int_{\R^N} G(x, w_n -u) \, dx \to 0.
$$
Hence $w_n \to u$ in $L^r (\R^N)$. From the continuous embedding $H^{\alpha / 2} (\R^N) \subset L^t (\R^N)$ for $t \in [1,2^*_\alpha]$, we know that $(w_n)$ is bounded in $L^t (\R^N)$ for every $1 \leq t \leq 2^*_\alpha$. In particular, $(w_n)$ is bounded in $L^2(\R^N)$ and in $L^{2^*_\alpha} (\R^N)$, so in view of H\"older inequality $w_n \to u$ in every $L^t(\R^N)$ for $t \in (2,2^*)$. 
Note that for every $\delta > 0$ there is $C_\delta > 0$ such that
\begin{eqnarray*}
& & \left| \int_{\R^N} f(x,w_n) (w_n - u) \, dx \right| \leq \delta \int_{\R^N} |w_n| |w_n-u| \, dx + C_\delta \int_{\R^N} |w_n|^{p-1} |w_n-u| \, dx \\
&\leq& \delta |w_n|_{2} | w_n - u|_{2} + C_\delta |w_n|_{p}^{p-1} |w_n - u|_{p} \\
&\to& \delta \limsup_{n\to\infty} \left( |w_n|_{2} |w_n-u|_{2} \right).
\end{eqnarray*}
Taking $\delta \to 0^+$ we obtain
$$
\left| \int_{\R^N} f(x,w_n) (w_n - u) \, dx \right| \to 0.
$$
Moreover
\begin{eqnarray*}
\left| \int_{\R^N} K_n (x) |w_n|^{q-2} w_n [w_n-u] \, dx \right| &\leq& |K_n|_\infty \int_{\R^N} |w_n|^{q-1} |w_n-u| \, dx \\
&\leq& |K_n|_\infty | w_n |^{q-1}_{q} | w_n - u |_{q} \to 0,
\end{eqnarray*}
since $(w_n)$ is bounded in $L^q (\R^N)$. Therefore
$$
\| w_n - u \|^2 = - \int_{\R^N} f(x,w_n)(w_n-u) \, dx +  \int_{\R^N} K_n(x) |w_n|^{q-2} w_n [w_n-u] \, dx + o(1) \to 0.
$$

\end{proof}

\section{Asymptotic behaviour of ground states as $\mu \to 0$}
\label{sect:AsymptoticMu}

\begin{proof}[Proof of Theorem \ref{ThAsymptoticMu}]
The proof is similar to proof of Theorem \ref{ThAsymptoticGamma} and \cite{GuoMederski}[Theorem 1.2], hence we provide only a sketch of the reasoning. Let $\cJ_n$ be the energy functional with $\mu = \mu_n$. For any $n \geq 1$ there is ground state $u_n$ in the corresponding Nehari manifold $\cN_n$. We also denote by $\cJ_0$ the energy with $\mu = 0$ and by $u_0 \in \cN_0$ the ground state for $\cJ_0$ in the corresponding Nehari manifold $\cN_0$. Similarly as in \cite{GuoMederski} we show that 
$$
\inf_{n \geq 1} \inf_{\|u\| = r} \cJ_n (u) > 0.
$$
Then we provide the following inequality
$$
c_n := \cJ_n (u_n) \geq \cJ_0 (u_0) + \frac{\mu_n}{2} \int_{\R^N} \frac{|t_n u_0|^2}{|x|^{\alpha}} \, dx \geq c_n + \frac{\mu_n}{2} \int_{\R^N} \frac{|t_n u_0|^2}{|x|^{\alpha}} \, dx,
$$
where $t_n > 0$ is such that $t_n u_n \in \cN_0$. Again, as in \cite{GuoMederski}, using the fractional Hardy inequality (Lemma \ref{Lem:HardyIneq}) we show that
$$
\frac{\mu_n}{2} \int_{\R^N} \frac{|t_n u_0|^2}{|x|^{\alpha}} \, dx \to 0.
$$
Hence $c_n \to c_0 := \cJ_0 (u_0)$. As in proof of Theorem \ref{ThAsymptoticGamma} we show that there is a sequence $(y_n) \subset \mathbb{Z}^N$ such that
$$
\liminf_{n\to\infty} \int_{B(y_n, 1+\sqrt{N})} |u_n|^2 \, dx > 0
$$
and
\begin{align*}
u_{n} (\cdot + y_{n}) \to u \quad &\mathrm{in} \ L^2_{loc} (\R^N), \\
u_{n} (\cdot + y_{n}) \weakto u \quad &\mathrm{in} \ H^{\alpha / 2} (\R^N), \\
u_{n} (x+y_{n}) \to u(x) \quad &\mathrm{for} \ \mathrm{a.e.} \ x \in \R^N.
\end{align*}
Using the fractional Hardy inequality and Lemma \ref{hardyLemma} we can repeat the reasoning from the proof of Theorem \ref{ThAsymptoticGamma} and show that $u_n (\cdot + y_n) \to u$ in $H^{\alpha / 2} (\R^N)$ and $u$ is a ground state for $\cJ_0$.
\end{proof}

\vspace{0.3cm}
\noindent
\textbf{Acknowledgement.} The author was partially supported by the National Science Centre, Poland (Grant No. 2017/25/N/ST1/00531).

\end{document}